\documentclass{amsart}
\usepackage{amsmath,amssymb,amsthm,amscd, mathtools, latexsym}
\usepackage{enumerate,varioref, dsfont, fancyhdr}
\usepackage{tikz-cd}
\usepackage{enumitem}
\usepackage{multicol}
\usepackage{hyperref}
\usepackage{url}
\usepackage{textcomp}

\newtheorem{Thm}{Theorem}[section]
\newtheorem{Lem}{Lemma}[section]
\newtheorem{Prop}{Proposition}[section]

\newtheorem{Cor}{Corollary}[section]
\newtheorem{Conj}{Conjecture}[section]
\newtheorem{Rem}{Remark}[section]
\newtheorem{Def}{Definition}[section]

{\theoremstyle{plain}

\newtheorem*{Ass*}{Assumption}

}
\renewcommand{\ker}{\text{ker}}

\def\cit{{\mathbb C}}

\def\qit{{\mathbb Q}}
\def\zit{{\mathbb Z}}

\def\0{{\mathcal O}}

\def\Hom{\mathop{\rm Hom}\nolimits}

\def\End{\mathop{\rm End}\nolimits}

\def\h{{\mathfrak h}}

\def\C{{\mathcal C}}

\def\h{{\mathfrak h}}

\def\M{{\mathcal M}}

\def\A{{\mathcal A}}

\begin{document}
\title[Some nilpotence theorems]{Some nilpotence theorems for algebraic cycles}

\author{H. Anthony Diaz}

\maketitle

\begin{abstract}
Using fundamental results of Deligne, we prove a nilpotence theorem for algebraic cycles and use this to prove a torsion nilpotence result for correspondences on surfaces.
 
\end{abstract}
\date{}

\section*{Introduction}

\noindent We let $k$ be a field and $\M_{k, R}$ be the category of Chow motives over $k$ with $R$ coefficients. For $X$ a smooth projective variety of dimension $d$ over $k$ and $\h(X) \in \M_{k, R}$ the corresponding Chow motive, recall that one has by definition 
\[\End_{\M_{k, R}} (\h(X)) = CH^{d} (X \times X)_{R}  \]
Unless otherwise specified, we will work with $\ell$-adic cohomology (for $\ell$ prime to the characteristic of $k$). For $X$ smooth and projective over $k$ there is a cycle class map: 
\begin{equation}CH^{r} (X)_{R} \to H^{2r}_{\text{\'et}} (X\times_{k} \overline{k}, R_{\ell}(r)))\label{cycle-class-map} \end{equation}
(where $R = \zit, \qit$), which gives homological equivalence for algebraic cycles, $\sim_{hom}$. Since the cycle class map is generally not injective, there are conjectures about the kernel, such as Voevodsky's smash-nilpotence conjecture for general codimension and the following nilpotence conjecture in codimension $d =dim(X)$ (see, for instance, \cite{J}):
\begin{Conj}\label{nil} For a Chow motive $M \in \M_{k, R}$ and $\Gamma \in \End_{\M_{k, R}} (M)$ such that $\Gamma \sim_{hom} 0$, there exists some $n>>0$ for which $\Gamma^{n} =0 \in \End_{\M_{k, R}} (M)$.
\end{Conj}
\noindent For $R = \qit$, the above conjecture is well-known and is of particular interest in determining whether or not the category of Chow motives possesses any {\em phantom motives} (i.e., Chow motives with vanishing cohomology). While several important conjectures in the theory of Chow motives imply this conjecture, it seems to be known only in the case of smooth projective varieties whose motives are summands of motives of Abelian varieties (see \cite{J}). Because of the intractibility of this problem, our aim will be to prove a nilpotence result which is substantially weaker. For this, we have the cycle class map:
\begin{equation}CH^{r} (X)_{R}\to H^{2r}_{\text{\'et}} (X, R_{\ell}(r))\label{cycle-class-map-k} \end{equation}
where again $R = \zit, \qit$. This is defined analogously to (\ref{cycle-class-map}) by means of the long exact sequence of a pair. Then, we have the following equivalence relation:
\begin{Def} Given $\gamma \in CH^{r} (X)_{R}$, we say that $\gamma$ is {\em $k$-homologically trivial with respect to $\ell$} and write $\gamma \sim_{k-hom} 0$ if $\gamma$ lies in the kernel of (\ref{cycle-class-map-k}), where $R= \zit, \qit$.
\end{Def}
\noindent We should note that a priori both $\sim_{hom}$ and $\sim_{k-hom}$ depends on the choice of $\ell$. On the other hand, for $R = \qit$ (which will be our primary interest here) the notion of nilpotence modulo $\sim_{hom}$ does not depend on $\ell$. When $\text{char } k =0$, this is clear. When $\text{char } k >0$, this follows from the fact that the characteristic polynomial of a correspondence is independent of $\ell$. Indeed, one quickly reduces to the case where $k$ is finitely generated over a finite field and, by a specialization argument, eventually one can to reduce to the case of a finite field, where it follows by \cite{KM} Theorem 2. 
\begin{Thm}\label{main} Let $X$ be a smooth projective variety of dimension $d$ over a field $k$.
\begin{enumerate}[label=(\alph*)]
\item\label{main-a} If $\Gamma \in \End_{\M_{k, \qit}}(\h(X))$ is homologically nilpotent, then $\Gamma$ is $k$-homologically nilpotent.
\item\label{main-b} Suppose $H^{*} (X_{\overline k}, \zit_{\ell})$ is torsion-free; suppose also that $k$ has characteristic $0$ or that $X$ satisfies the Lefschetz standard conjecture. If $\Gamma \in \End_{\M_{k}}(\h(X))$ is homologically nilpotent, then $\Gamma$ is $k$-homologically nilpotent.
\end{enumerate}
In fact, in both cases one can find some $N>>0$ for which $\Gamma^{N} \sim_{k-hom} 0$ for all $\ell$.
\end{Thm}
\begin{Cor} For any $\Gamma \in \End_{\M_{k, \qit}}(\h(X))$, there exists a polynomial $p[t] \in \qit[t]$ such that $p(\Gamma) \sim_{k-hom} 0$ with respect to every $\ell$.
\begin{proof} Let $q[t] \in \qit[t]$ be the characteristic polynomial of $[\Gamma]$ on $H^{*}_{\text{\'et}} (X_{\overline{k}}, \qit_{\ell})$. Then, set $\Gamma' := q(\Gamma)$ and we see that $\Gamma' \sim_{hom} 0$ with respect to every $\ell$. It follows from Theorem \ref{main} that there is some $N>>0$ for which $(\Gamma')^{N} \sim_{k-hom} 0$ with respect to every $\ell$.
\end{proof}
\end{Cor}

\noindent The key idea in the proof of Theorem \ref{main} is based on a fundamental decomposition result of Deligne for derived categories in his thesis \cite{D}, as well as his Hard Lefschetz theorem for $\ell$-adic cohomology. We can also use this idea to consider the nilpotence problem for torsion cycles. For $(n, \text{char } k) = 1$, consider the cycle class map for $n$-torsion cycles:
\begin{equation} CH^{d} (X \times X)[n] \to H^{2d}_{\text{\'et}} (X_{\overline{k}}\times X_{\overline{k}}, \zit/n(d))\label{n-tors-cycle}  \end{equation}
Then, we have the following torsion version of Conjecture \ref{nil}:
\begin{Conj}\label{tor-nil} If $\Gamma \in CH^{d} (X \times X)[n]$ is in the kernel of (\ref{n-tors-cycle}), then $\Gamma$ is a nilpotent correspondence.
\end{Conj}
\noindent We note, in particular, that if the $\ell$-adic cohomology of $X$ is torsion-free for $\ell \neq \text{char } k$, then this conjecture asks whether every $n$-torsion cycle is nilpotent for $(n, \text{char } k) = 1$. Since $CH^{d} (X \times X)$ has no torsion in the case that $X$ has dimension $1$, the question only applies in dimension $\geq 2$. In this direction, we have the following partial result:
\begin{Thm}\label{main-2} Suppose $X$ is a surface over a perfect field $k$. Then, Conjecture \ref{tor-nil} holds for $n = \ell^{r}$ for all but finitely many primes $\ell$.
\end{Thm}
\noindent The proof of this theorem will use Theorem \ref{main}, as well as a consequence of the Merkurjev-Suslin Theorem. When the codimension is $> 2$, the conjecture seems elusive even in instances for which Conjecture \ref{nil} is known. For instance, the $\ell^{r}$-torsion can be infinite, even when $X$ is an Abelian variety. In fact, one can take $X$ to be a triple product of very general elliptic curves over $\cit$; the main result of \cite{Di1} (or \cite{S2}) together with the main result of \cite{S1} show that the $\ell^{r}$-torsion is infinite in this case. It should be noted that Conjecture \ref{tor-nil} is a more general version of the Rost nilpotence conjecture, which predicts that cycles in the kernel of the extension of scalars functor:
\[ \End_{\M_{k}} (\h(X)) \to \End_{\M_{E}} (\h(X_{E})) \] 
(for $E/k$ a field extension) is nilpotent. This last conjecture is known for certain types of homogeneous varieties (see \cite{CGM}) and was recently established for surfaces (first by Gille in \cite{Gi} for fields of characteristic $0$ and then by Rosenschon and Sawant in \cite{RS} for perfect fields of characteristic $p$) but remains wide open in general.\\

\subsection*{Acknowledgements}
The author would like to thank Bruno Kahn for his interest and for taking the time to read several drafts of this paper; his comments were indispensable. He would also like to thank Robert Laterveer for some conversations which led to the formulation of the problems considered in the paper.
\section{Preliminaries} 
\subsection*{Some decomposition theorems} Let $X$ be a smooth projective variety over a field $k$ and let $\ell \neq \text{char } k$ and let $G_{k}$ denote the absolute Galois group. Then, there is a Hochschild-Serre spectral sequence for Jannsen's continuous \'etale cohomology (see \cite{J2}):
\[ H^{p}_{cont} (k, H^{q} (X_{\overline{k}}, \qit_{\ell})) \Rightarrow H^{p+q}_{\text{\'et}} (X_{\overline{k}}, \qit_{\ell}) \]
This spectral sequence degenerates by some non-trivial results of Deligne. More precisely, let $\pi: X \to \text{Spec } k$ be the structure morphism. Then, by Prop 1.2 of \cite{D} the degeneration of the above sequence is equivalent to a (non-canonical) isomorphism: 
\begin{equation} R_{cont}\pi_{*}\qit_{\ell} \cong \bigoplus_{i \in \zit} R_{cont}^{i}\pi_{*}\qit_{\ell}[-i] \label{deg-Q}\end{equation}
in $D^{b} (\text{Spec } k, \qit_{\ell})$, the bounded derived category of continuous $\qit_{\ell}[G_{k}]$-modules. (It ought to be noted that this is not a derived category in the usual sense, but the argument in \cite{D} still applies). Then, (\ref{deg-Q}) is obtained from the following result (together with Deligne's proof of the Hard Lefschetz theorem for $\ell$-adic cohomology):
\begin{Thm}[\cite{D} Thm. 1.5]\label{Deligne} Suppose that $\A$ is an Abelian category with bounded derived category $D^{b}(\A)$. For $A_{*} \in D^{b}(\A)$, there is an isomorphism:
\[ A_{*} \cong \bigoplus_{i \in \zit} H^{i} (A_{*})[-i] \in  D^{b}(\A) \]
if there is an integer $n$ and an endomorphism of degree $2$, $\phi: A_{*} \to A_{*}[2]$, which induces an isomorphism on cohomology:
\[ H^{n-i}(\phi^{i}): H^{n-i}(A_{*}) \to  H^{n+i}(A_{*}) \]
\end{Thm}
\noindent We can use this to prove an analogue of (\ref{deg-Q}) with integral coefficients which we will need for the proof of Theorem \ref{main-2}.
\begin{Prop}\label{lef-mod-tors} Suppse that $k$ has characteristic $0$ or that $X$ satisfies the Lefschetz standard conjecture. Then for all but finitely many primes $\ell$ there is a decomposition:
\[ R\pi_{*}\zit_{\ell} \cong \bigoplus_{i=0}^{2d} R^{i}\pi_{*}\zit_{\ell}[-i] \]
in $D^{b} (\text{Spec } k, \zit_{\ell})$, the bounded derived category of continuous $\zit_{\ell}[G_{k}]$-modules (in the sense of \cite{E}).
\begin{proof} By Theorem \ref{Deligne}, it suffices to prove that if $h \in Pic(X)$ is the class of an ample hyperplane divisor and $\mathcal{L}$ is the corresponding Lefschetz operator, then there is an isomorphism:
\[ H^{d-m}_{\text{\'et}} (X_{\overline{k}}, \zit_{\ell}) \xrightarrow[\cong]{\mathcal{L}^{m}} H^{d+m}_{\text{\'et}} (X_{\overline{k}}, \zit_{\ell}(m))\]
for all but finitely many primes $\ell$. Now, by \cite{Ga} $H^{*}_{\text{\'et}} (X_{\overline{k}}, \zit_{\ell})$ is torsion-free for all but finitely many $\ell$. So, it suffices to prove that the map:
\begin{equation} H^{d-j}_{\text{\'et}} (X_{\overline{k}}, \zit_{\ell})/tors \xrightarrow{\mathcal{L}^{m}} H^{d+j}_{\text{\'et}} (X_{\overline{k}}, \zit_{\ell})/tors \label{mod-tors}\end{equation}
is an isomorphism for all but finitely many $\ell$. In this direction, we note that since $X$ is defined over a finitely generated field, we can assume (by invariance of \'etale cohomology under algebraically closed extensions, \cite{M} VI Cor. 2.6) that $k$ is some finitely generately field. Using his proof of the Weil conjectures, Deligne proved in \cite{D2} the Hard Lefschetz theorem with $\qit_{\ell}$ coefficients for finite fields. Using the proper base change theorem (VI Cor. 2.3 of loc. cit.), one obtains the statement for all finitely generated fields of characteristic $p$. Thus, the Lefschetz operator $\mathcal{L}^{m}$ induces an isomorphism:
\[  H^{d-m}_{\text{\'et}} (X_{\overline{k}}, \qit_{\ell}) \xrightarrow[\cong]{\mathcal{L}^{m}} H^{d+m}_{\text{\'et}} (X_{\overline{k}}, \qit_{\ell}(m))\]
So, the map in (\ref{mod-tors}) is injective for all $\ell$ and so we only need to prove that it is surjective for all but finitely many $\ell$. In this direction, we note that if the Lefschetz standard conjecture holds, there exists:
\[ \Lambda \in CH^{d-1} (X \times X)_{\qit} \]
for which $\mathcal{L}^{m}\circ \Lambda^{m} = \text{ Id}$ on $H^{d+m}_{\text{\'et}} (X_{\overline{k}}, \qit_{\ell}(M))$. Let $M \in \zit$ be such that $M\cdot \Lambda \in CH^{d-1} (X \times X)$ so that setting $\Lambda' := M\cdot \Lambda$, we have
\[ \mathcal{L}^{m}\circ \Lambda^{m} = M\cdot\text{ Id} \]
on $H^{d+m}_{\text{\'et}} (X_{\overline{k}}, \zit_{\ell}(m))/tors$. Since $M$ is invertible for all but finitely many primes, it follows that the cokernel of (\ref{mod-tors}) is finite, as desired. On the other hand, if $k$ has characteristic $0$, one quickly reduces to the case of $k = \cit$, where one obtains 
\[ \Lambda \in H^{2(d-1)} (X \times X, \qit) \] 
in singular cohomology, and the remainder of the proof is then identical.
\end{proof}
\end{Prop}

\subsection*{Chow motives} For the applications below, we let $\M_{k, \qit}$ be the category of Chow motives over $k$ with rational coefficients. Recall that this is the pseudo-Abelian envelope of the additive category $Cor(k)$ whose objects are pairs $(X, m)$ where $X$ is smooth and projective of dimension $d$ and $m \in \zit$ and whose morphisms are given by correspondences:
\[ Cor ((X, m), (Y, m')) := CH^{d +m'-m} (X \times Y)_{\qit} \]
Thus, a Chow motive is given by a triple $(X, \pi, m)$, where $\pi \in Cor ((X, 0), (X, 0))$ is an idempotent. We also observe that $\M_{k, \qit}$ possesses the structure of a tensor category:
\[(X, \pi, m) \otimes (Y, \pi', m') := (X \times Y, \pi \times \pi', m+ m')\]
One can also define an internal symmetric and alternating product functor (using the usual Schur functor formalism). Moreover, there is the $\ell$-adic realization functor $R: \M_{k, \qit} \to D^{b}(\text{Spec } k, \qit_{\ell})$ which preserves the tensor product structure since for $\pi: X \to \text{Spec } k$, $\pi': Y \to \text{Spec } k$ smooth projective morphisms, we have
\[ R(\pi, \pi')_{*}\qit_{\ell, X \times Y} \cong R\pi_{*}\qit_{\ell, X} \boxtimes R\pi'_{*}\qit_{\ell, Y}  \]
by the K\"unneth theorem (\cite{M} VI.1). Finally, we will use the Tate twist notation
\[ M(i) = (X, \pi, m+i) \]
for $M= (X, \pi, m)$.

\section{Proof of Theorem \ref{main}}

\begin{proof}[Proof of Theorem \ref{main}]
The following proposition is key:
\begin{Prop} For a smooth projective morphism $\pi: X \to \text{Spec } k$, the kernel of the natural map of algebras:
\begin{equation} \End_{D} (R\pi_{*}A) \xrightarrow{\bigoplus_{i \in \zit} H^{i}} \bigoplus_{i \in \zit}\End_{A} (R^{i}\pi_{*}A)\label{main-map} \end{equation}
is a nil ideal, where $D$ is either the bounded derived category of continuous $\qit_{\ell}[G_{k}]$-modules (in the sense of \cite{J2}) with $A = \qit_{\ell}$ or the bounded derived category of continuous $ \zit_{\ell}[G_{k}]$-modules with $A = \zit_{\ell}$, provided that either of the assumptions of Prop. \ref{lef-mod-tors} are satisfied.
\begin{proof} We begin by noting that there is an obvious isomorphism:
\[\End_{A} (R^{i}\pi_{*}A) \cong \End_{D} (R^{i}\pi_{*}A[i]) \]
whose inverse is obtained by the functor $H^{i}$. Then, we set $I$ to be the kernel of (\ref{main-map}), and the proposition then follows from the lemma below (together with Theorem \ref{Deligne}), which was communicated to the author by Bruno Kahn and whose proof is essentially that of Prop. 2.3.4(c) from \cite{AKOS}.
\end{proof}
\end{Prop}
\begin{Lem}\label{key} Let $\C$ be an additive category and $C_{1}, C_{2}, \cdots C_{n} \in \C$ and set 
\[ C := \bigoplus_{i=1}^{n} C_{i} \]
Denote by $\iota_{i}: C_{i} \hookrightarrow C$ the inclusions and $\nu_{i}: C \to C_{i}$ the projections. Let $I$ be an ideal of $\End_{\C} (C)$ for which $I_{i} =\iota_{i}\circ I \circ \nu_{i}$ is a nil ideal with nilpotency index $r_{i}$. Then, $I$ is a nil ideal; in fact, $I^{N} =0$ for $N \geq r_{1} +r_{2} + \ldots +r_{n} + n-1$.
\begin{proof}[Proof of Lemma] By induction, we quickly reduce to the case of $n=2$. To this end, let $f \in I$ and write
\[ f = f_{11} + f_{12} + f_{21} + f_{22} \]
where $f_{ij} = \iota_{j}\circ \circ f \circ \nu_{i}$. Then, we compute
\begin{equation} f^{N} = (f_{11} + f_{12} + f_{21} + f_{22})^{N} = \sum_{I} a_{I}f_{i_{1}j_{1}}^{p_{1}}f_{i_{2}j_{2}}^{p_{2}}\ldots f_{i_{k}j_{k}}^{p_{k}}\label{summands}\end{equation}
where $I$ is the indexing set of all such products $f_{i_{1}j_{1}}^{p_{1}}f_{i_{2}j_{2}}^{p_{2}}\ldots f_{i_{k}j_{k}}^{p_{k}}$. Now, we observe that $\nu_{i}\iota_{j} = 0$ if $i \neq j$, so it follows that all terms in the above sum are $0$ except for those where
\[ p_{l} = 1 \text{ unless } i_{l}= j_{l}, \text{ and } i_{1} = j_{2}, \ldots i_{k-1} = j_{k} \]
Then, any summand in the sum (\ref{summands}) is expressible as:
\begin{equation} f_{22}^{n_{1}}f_{12}f_{11}^{m_{1}}f_{21}f_{22}^{n_{2}}\ldots f_{22}^{n_{k}}f_{12}f_{11}^{m_{k}}f_{21}f_{22}^{n_{k+1}} \label{type}
\end{equation}
or the corresponding product with the subscripts relabeled (i.e., $1$ and $2$ swapped). As in \cite{AKOS}, we see that this is contained in $I_{2}^{b} \cap I_{1}^{a}$ for $b \geq n_{1} + n_{2} + \ldots + n_{k+1} + k$ and $a \geq m_{1} + m_{2} + \ldots + m_{k} + k-1$. Since \[N = m_{1} + m_{2} + \ldots + m_{k} +n_{k} + m_{k+1} + 2k = a + b+1 \] 
We see then that taking $N = r_{1}+r_{2}+1$, then either $a \geq r_{1}$ or $b \geq r_{2}$, making $I_{2}^{b} \cap I_{1}^{a} = 0$. The lemma now follows.
\end{proof}
\end{Lem}

\noindent First we prove item \ref{main-a}. A basic observation is that the cycle class map coincides with the realization functor:
\[\End_{\M_{k, \qit}} (\h(X)) \to \End_{D^{b} (\text{Spec } k, \qit_{\ell})} (R\pi_{*}\qit_{\ell, X}) \]
after suitable identifation. In fact, we have the following chain of natural isomorphisms (see for instance Remark 4.2 of \cite{J3}):
\begin{equation} \begin{CD}
CH^{d} (X \times X)_{\qit} @>{}>> H^{2d}_{\text{\'et}} (X \times X, \qit_{\ell}(d))\\
@| @V{\cong}VV\\
\Hom_{\M_{k, \qit}} (\mathds{1}, \h(X\times X)[d]) @>>> \Hom_{D^{b} (\text{Spec } k, \qit_{\ell})} (\mathds{1}, R(\pi,\pi)_{*}\qit_{\ell, X\times X}(d)[2d])\\
@V{\cong}VV @V{\cong}VV\\
\Hom_{\M_{k, \qit}} (\mathds{1}, \h(X) \otimes \h(X)[d]) @>>> \Hom_{D^{b} (\text{Spec } k, \qit_{\ell})} (\mathds{1}, R\pi_{*}\qit_{\ell, X} \boxtimes R\pi_{*}\qit_{\ell, X} (d)[2d])\\
@V{\cong}VV @V{\cong}VV\\
\Hom_{\M_{k, \qit}} (\mathds{1}, \h(X) \otimes \h(X)^{\vee}) @>>> \Hom_{D^{b} (\text{Spec } k, \qit_{\ell})} (\mathds{1}, R\pi_{*}\qit_{\ell, X} \boxtimes R\pi_{*}\qit_{\ell, X}^{\vee})\\
@V{\cong}VV @V{\cong}VV\\
\End_{\M_{k, \qit}} (\h(X)) @>{}>> \End_{D^{b} (\text{Spec } k, \qit_{\ell})} (R\pi_{*}\qit_{\ell, X})\\
\end{CD}\label{big-commutative}
\end{equation}
where the first horizontal arrow is the cycle class map and the remaining horizontal arrows are the $\ell$-adic realization functors. On the left, the vertical isomorphisms follow from formal properties of Chow motives, while on the right, the vertical isomorphisms follow from the K\"unneth isomorphism, Poincar\'e duality and a property of the dual. Moreover, the bottom-most horizontal arrow respects composition (see, for instance, \cite{H} Prop. 19.2.4). There is also the ``extension-of-scalars" commutative diagram for cohomology:
\begin{equation}\begin{tikzcd} 
H^{2d}_{\text{\'et}} (X \times X, \qit_{\ell}(d)) \arrow{r}{\cong} \arrow{d} 
& \End_{D^{b} (\text{Spec } k, \qit_{\ell})} (R\pi_{*}\qit_{\ell, X}) \arrow{d}\\ 
H^{2d}_{\text{\'et}} (X_{\overline{k}} \times X_{\overline{k}}, \qit_{\ell}(d)) \arrow{r}{\cong} 
 & \bigoplus_{i=0}^{2d} \End_{\qit_{\ell}} (H^{i}_{\text{\'et}} (X_{\overline{k}}, \qit_{\ell})) \end{tikzcd}\label{commutative} \end{equation}
where the top horizontal arrow is the composition of right vertical arrows in (\ref{big-commutative}), the right vertical arrow is a special case of (\ref{main-map}) and the bottom arrow in (\ref{commutative}) is given by the composition of the following isomorphisms
\[ \begin{split}
H^{2d}_{\text{\'et}} (X_{\overline{k}} \times X_{\overline{k}}, \qit_{\ell}(d)) & \cong \bigoplus_{i=0}^{2d} H^{i}_{\text{\'et}}(X_{\overline{k}}, \qit_{\ell}(i))\otimes H^{2d-i}_{\text{\'et}}(X_{\overline{k}}, \qit_{\ell}(d-i))\\
& \cong \bigoplus_{i=0}^{2d} H^{i}_{\text{\'et}} (X_{\overline{k}}, \qit_{\ell}(i))\otimes H^{i}_{\text{\'et}} (X_{\overline{k}}, \qit_{\ell}(i))^{\vee}\\
& \cong \bigoplus_{i=0}^{2d} \End_{\qit_{\ell}} (H^{i}_{\text{\'et}} (X_{\overline{k}}, \qit_{\ell}))
\end{split} \]
Suppose then that the cycle class of
\[ \Gamma \in \End_{\M_{k, \qit}} (\h(X)) \]
is nilpotent in $\bigoplus_{i=0}^{2d} \End_{\qit_{\ell}} (H^{i}_{\text{\'et}} (X_{\overline{k}}, \qit_{\ell}))$. From the last proposition, it then follows that \[[\Gamma^{N}] = [\Gamma]^{N} = 0 \in H^{2d}_{\text{\'et}} (X \times X, \qit_{\ell}(d)) \cong End_{D^{b} (\text{Spec } k, \qit_{\ell})} (Rf_{*}\qit_{\ell, X})\] for some $N>>0$ not depending on $\ell$. An inspection of the above proof shows that it works mutatis mutandis to prove item \ref{main-b}, so long as $H^{*}_{\text{\'et}} (X_{\overline{k}}, \zit_{\ell}(r))$ is torsion-free. 
\end{proof}

\section{Proof of Theorem \ref{main-2}}
\noindent Suppose first that 
\[ \Gamma \in \ker \{CH^{2} (X \times X) \to CH^{2} (X_{\overline{k}} \times X_{\overline{k}}) \} \]
Then, the Rost nilpotence conjecture holds (see \cite{Gi} or \cite{RS}), so it suffices to prove that 
\[ \Gamma \in CH^{2} (X_{\overline{k}} \times X_{\overline{k}})[\ell^{\infty}] \]
is nilpotent, which we can reduce to the case of $k$ finitely generated using Suslin rigidity in the form of \cite{Le}. For ease of notation, write $Y = X \times X$. Then, we have the cycle class map:
\begin{equation} CH^{2} (Y_{\overline{k}}) :=  \mathop{\lim_{\longrightarrow}}_{[k': k] <\infty} CH^{2} (Y_{k'}) \rightarrow \mathop{\lim_{\longrightarrow}}_{[k': k] <\infty} H^{4}_{\text{\'et}} (Y_{k'}, \zit_{\ell}(2))\label{cycle}\end{equation}
Now, we let
\[ H^{4}_{\text{\'et}} (Y_{k'}, \zit_{\ell}(2))_{0} := \ker \{ H^{4}_{\text{\'et}} (Y_{k'}, \zit_{\ell}(2)) \to H^{4}_{\text{\'et}} (Y_{\overline{k}}, \zit_{\ell}(2))^{G_{k'}} \}\]
For all but finitely many $\ell$, $H^{4}_{\text{\'et}} (Y_{\overline{k}}, \zit_{\ell}(2))$ is torsion-free by \cite{Ga}, so for all but finitely many such $\ell$ the image of an $\ell$-ary torsion cycle under (\ref{cycle}) lies in 
\[\mathop{\lim_{\longrightarrow}}_{[k': k] <\infty} H^{4}_{\text{\'et}} (Y_{k'}, \zit_{\ell}(2))_{0} \]
and thus we can consider the Abel-Jacobi map:
\[ CH^{2} (Y_{\overline{k}})[\ell^{\infty}] \to \mathop{\lim_{\longrightarrow}}_{[k': k] <\infty} H^{4}_{\text{\'et}} (Y_{k'}, \zit_{\ell}(2))_{0} \xrightarrow{\text{AJ}} \mathop{\lim_{\longrightarrow}}_{[k': k] <\infty} H^{1}(G_{k'}, H^{3}_{\text{\'et}} (Y_{\overline{k}}, \zit_{\ell}(2)))   \]
There is also the (direct limit of the) boundary map:
\[ \partial: H^{3}_{\text{\'et}} (Y_{\overline{k}}, \qit_{\ell}/\zit_{\ell}(2)) = \mathop{\lim_{\longrightarrow}}_{[k': k] <\infty} H^{3}_{\text{\'et}} (Y_{\overline{k}}, \qit_{\ell}/\zit_{\ell}(2))^{G_{k'}} \to \mathop{\lim_{\longrightarrow}}_{[k': k] <\infty}  H^{1}(G_{k'}, H^{3}_{\text{\'et}} (Y_{\overline{k}}, \zit_{\ell}(2))) \]
obtained by applying $G_{k'}$-invariants to the short exact sequence of $G_{k}$-modules (obtained from Bockstein using the fact that there is no torsion in $H^{*}_{\text{\'et}} (Y_{\overline{k}}, \zit_{\ell}(2))$) and then taking a direct limit over $r$:
\[ 0 \to H^{3}_{\text{\'et}} (Y_{\overline{k}}, \zit_{\ell}(2)) \xrightarrow{\times \ell^{r}} H^{3}_{\text{\'et}} (Y_{\overline{k}}, \zit_{\ell}(2)) \to H^{3}_{\text{\'et}} (Y_{\overline{k}} , \zit_{\ell}/\ell^{r}(2)) \to 0\]
Now, our key observation is that these maps fit into the following commutative diagram:
\begin{equation} \begin{CD}
CH^{2} (Y_{\overline{k}}) [\ell^{\infty}] @>{cl_{B}}>> H^{3}_{\text{\'et}} (Y_{\overline{k}}, \qit_{\ell}/\zit_{\ell}(2))\\
@VVV @V{\partial}VV\\
 \displaystyle 
\mathop{\lim_{\longrightarrow}}_{[k': k] <\infty} H^{4}_{\text{\'et}} (Y_{k'}, \zit_{\ell}(2))_{0} @>{\text{AJ}}>> \displaystyle\mathop{\lim_{\longrightarrow}}_{[k': k] <\infty} H^{1} (G_{k'}, H^{2}_{\text{\'et}} (Y_{\overline{k}}, \zit_{\ell}(2)))
\end{CD}\label{composition}
\end{equation}
where $cl_{B}$ is Bloch's cycle class map (see \cite{B}).
\begin{proof}[Proof that (\ref{composition}) commutes] Let $Y$ any smooth projective variety over $k$ and assume for simplicity that $H^{*}_{\text{\'et}} (Y_{\overline{k}}, \zit_{\ell})$ is torsion-free. Now, suppose that $\gamma \in CH^{2} (Y_{\overline{k}}) [\ell^{\infty}]$. Then, there exists some finite Galois extension $k'\subset k$ over which $\gamma$ is defined. In particular, 
\[ cl_{B} (\gamma) \in H^{3}_{\text{\'et}} (Y_{\overline{k}}, \qit_{\ell}/\zit_{\ell}(2))^{G_{k'}} \]
Then, there is a surjective map:
\begin{equation} H^{1}_{\text{Zar}} (Y_{k'}, \mathcal{K}_{2, \qit_{\ell}/\zit_{\ell}}) \to H^{2}_{\text{Zar}} (Y_{k'}, \mathcal{K}_{2}) [\ell^{\infty}] = CH^{2} (Y_{k'}) [\ell^{\infty}]\label{surj}\end{equation}
where $\mathcal{K}_{2}$ is the Zariski sheaf associated to the presheaf of the Quillen $K$-theory of $Y$ (similarly, $\mathcal{K}_{2, \qit_{\ell}/\zit_{\ell}}$ is the direct limit of the torsion sheaves $\mathcal{K}_{2, \zit/\ell^{r}}$). Then, the Galois symbol induces maps:
\begin{equation} \begin{split} H^{i}_{Zar} (Y_{k'}, \mathcal{K}_{2, \qit_{\ell}/\zit_{\ell}}) & \to H^{i}_{Zar} (Y_{k'}, \mathcal{H}^{2}_{\qit_{\ell}/\zit_{\ell}}(2))\\
H^{i}_{Zar} (Y_{k'}, \mathcal{K}_{2}) & \to H^{i}_{Zar} (Y_{k'}, \mathcal{H}^{2}_{\zit_{\ell}}(2))\\
\end{split}\label{double} \end{equation}
where $\mathcal{H}^{2}_{R}(2)$ denotes the Zariski sheafification of $U \mapsto H^{2}_{\text{\'et}} (U, R(2))$ for $R = \qit_{\ell}/\zit_{\ell}, \zit_{\ell}$. We observe that $i=2$ in (\ref{double}) gives the cycle class map. Moreover, (\ref{surj}) fits into a commutative diagram:
\begin{equation} \begin{CD}
H^{1}_{\text{Zar}} (Y_{k'}, \mathcal{K}_{2, \qit_{\ell}/\zit_{\ell}}) @>>> H^{2}_{\text{Zar}} (Y_{k'}, \mathcal{K}_{2}) [\ell^{\infty}]\\
@VVV @VVV\\
H^{1}_{\text{Zar}} (Y_{k'}, \mathcal{H}^{2}_{\qit_{\ell}/\zit_{\ell}}(2)) @>>> H^{2}_{\text{Zar}} (Y_{k'}, \mathcal{H}^{2}_{\zit_{\ell}}(2)) [\ell^{\infty}]\\
@VVV @VVV\\
H^{3}_{\text{\'et}} (Y_{k'}, \qit_{\ell}/\zit_{\ell}(2)) @>>> H^{4}_{\text{\'et}} (Y_{k'}, \zit_{\ell}(2))[\ell^{\infty}]
\end{CD}\label{big-CD}
\end{equation}
Here, the top vertical arrows are those of (\ref{double}), the bottom vertical arrows arise from the spectral sequence:
\[H^{p}_{Zar} (Y_{k'}, \mathcal{H}^{q}_{R}(2)) \Rightarrow H^{p+q}_{\text{\'et}} (Y_{k'}, R(2)) \]
and the bottom horizontal arrow is the boundary map of the Bockstein exact sequence:
\[ 0 \to H^{3}_{\text{\'et}} (Y, \zit_{\ell}(2))\otimes \qit_{\ell}/\zit_{\ell} \to H^{3}_{\text{\'et}} (Y, \qit_{\ell}/\zit_{\ell}(2)) \to H^{4}_{\text{\'et}} (Y, \zit_{\ell}(2))[\ell^{\infty}] \to 0 \]
Of course, there is a diagram analogous to (\ref{big-CD}) with $k'$ replaced by $\overline{k}$. In fact, using the Weil conjectures, Bloch derives $cl_{B}$ as the dotted arrow in:
\begin{equation}\begin{tikzcd}
H^{1}_{\text{Zar}} (Y_{\overline{k}}, \mathcal{K}_{2, \qit_{\ell}/\zit_{\ell}}) \arrow{r} \arrow{d} & CH^{2} (Y_{\overline{k}}) [\ell^{\infty}] \arrow{d} \arrow[dotted]{ld}\\
H^{3}_{\text{\'et}} (Y_{\overline{k}}, \qit_{\ell}/\zit_{\ell}(2)) \arrow{r} & H^{4}_{\text{\'et}} (Y_{\overline{k}}, \zit_{\ell}(2)) [\ell^{\infty}]
\end{tikzcd}\label{small-CD}
\end{equation}
Further, these (\ref{big-CD}) and (\ref{small-CD}) fit into a commutative cube in which the connecting arrows are the obvious extension-of-scalars maps from $k'$ to $\overline{k}$.
Thus, there is some $\tilde{\gamma} \in H^{1}_{\text{Zar}} (Y_{k'}, \mathcal{K}_{2, \qit_{\ell}/\zit_{\ell}})$ whose image under
\[H^{1}_{\text{Zar}} (Y_{k'}, \mathcal{K}_{2, \qit_{\ell}/\zit_{\ell}}) \to H^{3}_{\text{\'et}} (Y_{k'}, \qit_{\ell}/\zit_{\ell}(2)) \to H^{3}_{\text{\'et}} (Y_{\overline{k}}, \qit_{\ell}/\zit_{\ell}(2))^{G_{k'}} \]
coincides with $cl_{B}(\gamma)$.
Thus, what remains is to observe that there is a commutative diagram:
\begin{equation}\begin{CD}
H^{3}_{\text{\'et}} (Y_{k'}, \qit_{\ell}/\zit_{\ell}(2)) @>>> H^{4}_{\text{\'et}} (Y_{k'}, \zit_{\ell}(2))[\ell^{\infty}]\\
@VVV @V{AJ}VV\\
H^{3}_{\text{\'et}} (Y_{\overline{k}}, \qit_{\ell}/\zit_{\ell}(2))^{G_{k'}} @>{\partial}>>  H^{1}(G_{k'}, H^{3}_{\text{\'et}} (Y_{\overline{k}}, \zit_{\ell}(2)))[\ell^{\infty}]
\end{CD}\label{key-CD}\end{equation}
To see this, note that the vertical maps arise from the Hochschild-Serre spectral sequence (the right map being well-defined since $H^{4}_{\text{\'et}} (Y_{\overline{k}}, \zit_{\ell}(2))$ is torsion-free). Then, there is a short exact sequence in the Ekedahl category of \'etale sheaves over $Y_{k'}$:
\[ 0 \to \zit_{\ell}(2) \to \qit_{\ell}(2) \to \qit_{\ell}/\zit_{\ell}(2) \to 0 \]
Let $\pi: Y_{k'} \to \text{Spec } k'$ be the structure morphism; applying $R\pi_{*}$ to this gives the boundary map:
\[ R\pi_{*}\qit_{\ell}/\zit_{\ell}(2) \to R\pi_{*}\zit_{\ell}(2) [1] \] 
which induces the horizontal maps in (\ref{key-CD}). There is a corresponding map of spectral sequences; for the $E_{2}$-page, this means there are commutative diagrams:
\[\begin{CD}
H^{p} (G_{k'}, H^{q}_{\text{\'et}} (Y_{\overline{k}}, \qit_{\ell}/\zit_{\ell}(2)) @>>> H^{p+1} (G_{k'}, H^{q}_{\text{\'et}} (Y_{\overline{k}}, \qit_{\ell}/\zit_{\ell}(2))\\
@V{d_{2}^{p,q}}VV @V{d_{2}^{p+1,q}}VV\\
H^{p+2} (G_{k'}, H^{q-1}_{\text{\'et}} (Y_{\overline{k}}, \qit_{\ell}/\zit_{\ell}(2)) @>>> H^{p+3} (G_{k'}, H^{q-1}_{\text{\'et}} (Y_{\overline{k}}, \qit_{\ell}/\zit_{\ell}(2))
\end{CD} 
\]
as well as for the other pages. We deduce then that (\ref{key-CD}) commutes, as desired.
\end{proof}
\begin{Rem} We should note that the commutativity of (\ref{composition}) is essentially an algebraic version of Prop. 3.7 in \cite{B}.
\end{Rem}

\noindent To finish the proof, we observe that Bloch's cycle class map is injective by a consequence of the Merkurjev-Suslin Theorem (see \cite{MS}). Moreover, the map $\partial$ in (\ref{composition}) is an isomorphism because of a weights argument given in \cite{S3} Lemma 1.1.3. By (\ref{composition}), it follows that the direct limit of cycle class maps:
\[CH^{2} (X_{\overline{k}} \times X_{\overline{k}})[\ell^{\infty}] \to \mathop{\lim_{\longrightarrow}}_{[k': k] <\infty} H^{4}_{\text{\'et}} (X_{k'} \times X_{k'}, \zit_{\ell}(2))_{0} \]
is injective. Since surfaces satisfy the Lefschetz standard conjecture, it follows from Theorem \ref{main} \ref{main-b} that $\Gamma$ is nilpotent in the right group and, hence, also in the left by injectivity.

\end{document}